\documentclass{amsart}

\usepackage[foot]{amsaddr}

\usepackage{subfig}
\usepackage{graphicx}
\usepackage{amssymb,amsfonts,amsrefs}
\usepackage{hyperref}
\usepackage{tikz-cd}
\usepackage{enumerate}
\usepackage{array}
\usepackage{changes}

\usepackage[all,cmtip]{xy}

\usepackage{bbm}
\usepackage{tikz-cd}
\usepackage{mathtools}
\usetikzlibrary{calc}

\usepackage{comment}
\usepackage{mathcommands}

\makeatletter
\newtheorem*{rep@theorem}{\rep@title}
\newcommand{\newreptheorem}[2]{%
\newenvironment{rep#1}[1]{%
 \def\rep@title{#2 \ref{##1}}%
 \begin{rep@theorem}}%
 {\end{rep@theorem}}}
\makeatother

\newtheorem{thm}{Theorem}[section]
\newreptheorem{thm}{Theorem}
\newtheorem*{thm*}{Theorem}
\newtheorem{cor}[thm]{Corollary}
\newtheorem{prop}[thm]{Proposition}
\newtheorem{lem}[thm]{Lemma}
\newtheorem{conj}[thm]{Conjecture}
\newtheorem*{conj*}{Conjecture}
\newtheorem{quest}[thm]{Question}

\newtheorem{assumpt}[thm]{Assumption}

\newtheorem{mainthm}{Theorem}

\theoremstyle{definition}
\newtheorem{defn}[thm]{Definition}

\newtheorem{con}[thm]{Construction}
\newtheorem{exmp}[thm]{Example}

\theoremstyle{remark}
\newtheorem{rem}[thm]{Remark}

\newcommand{\rmH}{\mathrm{H}}

\numberwithin{equation}{section}

\title{On the  Bogomolov-Positselski Conjecture
}

\author{Julian Feuerpfeil}

\address{Dipartimento di Matematica e Applicazioni, Università di Milano-Bicocca}

\address{Femto-St, Université Marie et Louis Pasteur}

\email{\href{mailto:j.feuerpfeil@campus.unimib.it}{j.feuerpfeil@campus.unimib.it}}

\begin{document}

\begin{abstract} 
Let $p$ be a prime. An oriented pro-$p$ group $(G,\theta)$ is said to have the Bogomolov--Positselski property if it is Kummerian and if $I_\theta(G)$ is a free pro-$p$ group. In this paper, we provide a new criterion for an oriented pro-$p$ group to satisfy the Bogomolov--Positselski property. This criterion builds on earlier work of Positselski~\cite{Positselski2005} and Quadrelli--Weigel~\cite{QuadrelliWeigel2022}, relates their approaches, and answers a question raised in~\cite{QuadrelliWeigel2022}.

Under additional assumptions, we obtain two further sufficient criteria. The first is analogous to a Merkurjev--Suslin type statement. The second allows one to weaken the hypotheses appearing in Positselski's criterion~\cite{Positselski2005}*{Theorem~2}. Finally, we show that the stronger conditions are satisfied by pro-$p$ groups of elementary type. As a consequence, the Elementary Type Conjecture implies Positselski's \lq\lq Module Koszulity Conjecture~1\rq\rq~\cite{Positselski2014} for fields with finitely generated maximal pro-$p$ Galois group.



\noindent \textbf{ORCID:} \ \href{https://orcid.org/0009-0000-0148-3348}{0009-0000-0148-3348}

\end{abstract}

\maketitle

\section{Introduction}
\subsection{Oriented pro-\texorpdfstring{$p$}{} groups and maximal pro-\texorpdfstring{$p$}{} Galois groups}
Let $p$ be a prime. A $p$-\emph{oriented profinite group} is a pair $(G,\theta)$ consiting of a profinite group $G$ and continuous homomorphism $\theta:G\to \ZZ_p^\times$. They have been introduced by I. Efrat for pro-$p$ groups in \cite{Efrat1998} under the name \emph{cyclotomic pro-$p$ pair}. In contrast, we call a $p$-oriented profinite group, whose underlying profinite group is pro-$p$, an \emph{oriented pro-$p$ group}, as done by Quadrelli and Weigel in \cite{QuadrelliWeigel2022}.

The definition is motivated by a setting in Galois theory. Let $\mathbb{K}$ be a field of characteristic $\neq p$, denote by $\mathbb{K}^s$ a separable closure of $\mathbb{K}$ and $G_\mathbb{K}:=\Gal(\mathbb{K}^s/\mathbb{K})$ the absolute Galois group of $\mathbb{K}$. The profinite group $G_\mathbb{K}$ acts continuously on the discrete group $\smash{\mu_{p^\infty}(\mathbb{K}^s)}\cong \QQ_p/\ZZ_p$. This action defines a continuous homomorphism 
\begin{align}
\theta_{\KK}:G_{\mathbb{K}}\to \Aut(\mu_{p^\infty}(\mathbb{K}^s))\cong \ZZ_p^\times
\end{align}
and the pair $(G_\mathbb{K},\theta_{\KK})$ is a $p$-oriented profinite group. If $\mathbb{K}$ contains a primitive $p^{\text{th}}$ root of unity, then $\theta_{\KK}$ factors through the maximal pro-$p$ quotient $G_\mathbb{K}(p):=G_\mathbb{K}/O^p(G_\mathbb{K})$, where $O^p(G_\mathbb{K})$ is the normal subgroup generated by all $p'$-Sylow subgroups of $G_\mathbb{K}$. In Galois-theoretic terms, this quotient corresponds to the maximal pro-$p$ Galois group of $\mathbb{K}$. We also denote the induced orientation on $G_\mathbb{K}(p)$ by $\theta_{\KK}$.

An oriented pro-$p$ group $(G,\theta)$ is called \emph{torsion-free}, $p$ is odd or $p=2$ and $\theta(G)\subseteq 1+4\ZZ_2$. In this case $\theta(G)$ is isomorphic to $\ZZ_p$ or trivial. Notice that this does not imply that $G$ itself is torsion-free as a pro-$p$ group. The oriented pro-$p$ group $(G_\mathbb{K}(p),\theta_{\KK})$ is torsion-free if and only if $p$ is odd or $p=2$ and $\sqrt{-1}\in \mathbb{K}$.

An oriented pro-$p$ group $(G,\theta)$ contains apart from $\ker \theta$ the following distinguished closed subgroups:
\begin{align*}
    K_\theta(G)&:=\langle h^{-\theta(g)}ghg^{-1}:g\in G, h\in \ker \theta\rangle_{cl}\\
    I_\theta(G)&:=\langle h\in \ker \theta:h^{p^n}\in K_\theta(G)\text{ for some }n\in \NN_0\rangle_{cl}
\end{align*}
The normal subgroup $K_\theta(G)$ was introduced by Quadrelli and Efrat in \cite{EfratQuadrelli2019} and is an analogue of the commutator subgroup for oriented pro-$p$ groups (see \cite{QuadrelliWeigel2022}). The subgroup $I_\theta(G)$ is normal and the \emph{isolator} of $K_\theta(G)$ in $G$. If $G$ is clear from the context, we occasionally simply write $K_\theta$ and $I_\theta$. The quotient $G(\theta):=G/I_\theta(G)$ is the maximal $\theta$-abelian quotient of $G$ (see \cite{QuadrelliWeigel2022}*{Section 2}). If $\theta$ is trivial, then $G(\theta)$ is the maximal torsion-free quotient of $G^{\rm ab}$.

An oriented pro-$p$ group $(G,\theta)$ is called \emph{Kummerian} if $\ker(\theta)/K_\theta(G)$ is a free abelian pro-$p$ group. This notion was introduced by Efrat and Quadrelli in \cite{EfratQuadrelli2019}*{Definition 3.4} and has proven to be a powerful tool to exclude oriented pro-$p$ groups as candidates for maximal pro-$p$ Galois groups with cyclotomic orientation (see for example \cite{EfratQuadrelli2019}*{Section 8}). There are many equivalent characterizations of the Kummerian property (see, for example, \cite{QuadrelliWeigel2022}*{Proposition 2.6}). One of them is that $(G,\theta)$ is Kummerian if and only if $K_\theta(G)=I_\theta(G)$.

\begin{thm*}[{\cite{EfratQuadrelli2019}*{Theorem 4.2}}]
    Let $\mathbb{K}$ be a field containing a primitve $p^{\text{th}}$ root of unity (and $\sqrt{-1}$ if $p=2$), then $(G_\mathbb{K}(p),\theta_{\KK})$ is a torsion-free, Kummerian oriented pro-$p$ group.
\end{thm*}

Most of the statements in this paper are only concerned (and only true) for torsion-free oriented pro-$p$ groups. To ensure the validity in the Galois theoretic context, we make the following standing assumption:
\begin{assumpt}
    \label{ass:roots_of_unity}
    The field $\mathbb{K}$ contains a primitive $p^{\text{th}}$ root of unity and $\sqrt{-1}$ if $p=2$.
\end{assumpt}

The following conjecture was first stated by Bogomolov in \cite{Bogomolov1995} for fields containing an algebraically closed subfield and later refined by Positselski in \cite{Positselski2005} to fields statisfying \ref{ass:roots_of_unity}:
\begin{conj*}[Bogomolov--Positselski]
\label{conj:BogomolovPositselski}
    Let $\mathbb{K}$ be a field satisfying \ref{ass:roots_of_unity}, then the group $\smash{K_{\theta_{\KK}}(G_\mathbb{K}(p))}$ is a free pro-$p$ group. Equivalently, the maximal pro-$p$ Galois group of 
    \begin{align*}
        \sqrt[p^\infty]{\mathbb{K}}:=\mathbb{K}(\sqrt[p^n]{a}:n\in \NN,a\in \mathbb{K})
    \end{align*} is a free pro-$p$ group.
\end{conj*}
Motivated by this Conjecture, we say that a Kummerian oriented pro-$p$ group $(G,\theta)$ has the \emph{Bogomolov--Positselski property} if $K_\theta(G)$ is a free pro-$p$ group. Similarly, a field $\mathbb{K}$ satisfying \ref{ass:roots_of_unity} has the \emph{Bogomolov--Positselski property} if $(G_\mathbb{K}(p),\theta_{\KK})$ has the Bogomolov--Positselski property.

A pro-$p$ group $G$ is called $\rmH^\bullet$-\emph{quadratic} if its $\FF_p$-cohomology algebra $\rmH^\bullet(G,\FF_p)=\bigoplus_i \rmH^i(G,\FF_p)$ is a quadratic algebra with respect to the cup product, that is, it is generated as algebra by its elements in degree $1$ and all relations are in degree $2$. For a more precise definition, we refer to Section~\ref{sec:QuadraticAndKoszulAlgebras}.

The following theorem is a consequence of the norm residue isomorphism theorem proven by Rost and Voevodsky with a patch by Weibel (cf.  \cites{Voevodsky2011,Weibel2008,Weibel2009}):
\begin{thm*}
    \label{thm:NormResidueIso}
    Let $\mathbb{K}$ be a field containing a primitive $p^{\text{th}}$ root of unity, then $G_\mathbb{K}(p)$ is $\rmH^\bullet$-quadratic. 
\end{thm*}

For a torsion-free, Kummerian, oriented pro-$p$ group $(G,\theta)$, we consider the inflation map
\begin{align}
    \psi_G^\bullet:={\inf}_{G(\theta),G}^\bullet:\rmH^\bullet(G(\theta),\FF_p)\cong \Lambda^\bullet(\rmH^1(G,\FF_p))\to \rmH^\bullet(G,\FF_p),
\end{align}
which is surjective homomorphism of quadratic algebras if $G$ is $\rmH^\bullet$-quadratic. The isomorphism $\rmH^\bullet(G(\theta),\FF_p)\cong \Lambda^\bullet(\rmH^1(G,\FF_p))$ can be found in \cite{QuadrelliWeigel2022}*{Example 4.3} and is a consequence of Lazard's theorem. If $G$ is clear from the context, we only write $\psi$ instead of $\psi_G$.

The next theorem is due to Positselski and gives a criterion for the Bogomolov--Positselski property of a field $\KK$ in terms of properties of the kernel of $\psi^\bullet:=\psi_{G_\KK(p)}^\bullet$, whose proof works also in the purely group theoretic setting. 
\begin{thm*}[{\cite[Theorem 2]{Positselski2005}}]
\label{thm:Positelski_Koszul}
    Let $\mathbb{K}$ be a field satisfying \ref{ass:roots_of_unity}. If $(\ker \psi^\bullet)(2)$ is a Koszul module over the algebra $\Lambda^\bullet(\rmH^1(G_\mathbb{K},\FF_p))$, then $\mathbb{K}$ has the Bogomolov--Positselski property.
\end{thm*}
This criterion depends on the vanishing of infinitely many cohomology groups, since the definition of Koszulity asserts that $\rmH_{ij}(\Lambda^\bullet(V),\ker \psi^\bullet)=0$ for all $j\neq i+2$, where $V=\rmH^1(G_\mathbb{K},\FF_p)$. See Section~\ref{sec:CohomologyOfGradedAlgebras} for the definition of the (co-)homology groups of graded algebras. Positselski conjectured that the conditions for this theorem hold universally in \cite{Positselski2014}*{Conjecture}.

In \cite{QuadrelliWeigel2022}, Quadrelli and Weigel gave a new criterion for the Bogomolov--Positselski property, depending only on two cohomology groups, but in a sophisticated way. Let $(G,\theta)$ be a torsion-free Kummerian oriented pro-$p$ group, then there is the Hochschild-Serre spectral sequence associated to the group extension $1\to K_\theta(G)\to G\to G(\theta)\to 1$.
This spectral sequence will be denoted by 
\begin{align}
\label{equ:HochschildSerreSpectral}
    E_2^{s,t}:=\rmH^s(G(\theta),\rmH^t(K_\theta(G),\FF_p))\Longrightarrow \rmH^{s+t}(G,\FF_p).
\end{align}
\begin{thm*}[{\cite[Theorem 4.5]{QuadrelliWeigel2022}}]
\label{thm:QuadrelliWeigel_BPC}
    Let $(G,\theta)$ be a torsion-free, Kummerian, oriented pro-$p$ group with $G$ being $\rmH^\bullet$-quadratic, then $(G,\theta)$ has the Bogomolov--Positselski property if and only if the differential $d_2^{2,1}:E^{2,1}_2\to E^{4,0}_2$ in the spectral sequence in (\ref{equ:HochschildSerreSpectral}) is injective.
\end{thm*}

\subsection{Main results and structure of the paper}
Quadrelli and Weigel asked in \cite[Remark 1.5]{QuadrelliWeigel2022} if there is a connection between Theorem their theorem and the criterion by Positsitselski. More precisely, if there is a way to express $\ker d_2^{2,1}$ in terms of the certain (co-)homology groups of graded $\rmH^\bullet(G(\theta),\FF_p)$ modules.

In this paper, we give an affirmative answer to this question. The following theorem gives a first description and is an important cornerstone to the other criteria.
\begin{mainthm}
\label{thm:TorSES}
    Let $(G,\theta)$ be a torsion-free, Kummerian, oriented pro-$p$ group, such that $G$ is $\rmH^\bullet$-quadratic. We set  $V:=\rmH^1(G,\FF_p)$, $B:=\rmH^\bullet(G,\FF_p)$ and let $N$ be the graded $\Lambda^\bullet(V)$-module $\rmH^\bullet(G(\theta),\rmH^1(K_\theta,\FF_p))$. Then there is an exact sequence:
    \begin{equation*}
        \begin{tikzcd}[column sep=small]
            0\arrow[r]&\rmH_{1,2}(\Lambda^\bullet(V),N)\arrow[r] &\rmH_{2,4}(\Lambda^\bullet(V),B)\arrow[r] &\ker d_2^{2,1}\arrow[r]&\rmH_{0,2}(\Lambda^\bullet(V),N)\arrow[r]&0
        \end{tikzcd}
    \end{equation*}
    In particular $\ker d_2^{2,1}=0$ if and only if the first map is an isomorphism and $H_{0,2}(\Lambda^\bullet(V),N)=0$.
\end{mainthm}
This theorem depends on only finitely many (co)homology groups, but again in a sophisticated way, as the module $N$, whose properties as $\rmH^\bullet(G(\theta),\FF_p)$-module determine the Bogomolov--Positselski property, seems to be hard to control. Nevertheless, the vanishing of $H_{0,2}(\Lambda^\bullet(V),N)$ has a concrete description leading to Corollary~\ref{cor:CupProductSurjective}, which has striking similarity with the statement of the Merkurjev--Suslin theorem. 

On the other hand, the vanishing of the group $H_{2,4}(\Lambda^\bullet(V),B)\cong H_{1,4}(\Lambda^\bullet(V),\ker \psi^\bullet)$ is predicted by Positselski's Module Koszulity Conjecture~1 and follows from even weaker properties already. This leads to the following question:
\begin{quest}
    \label{ques:IsTor24eq0}
    What conditions on an oriented pro-$p$ group $(G,\theta)$ are sufficient in order to conclude $\rmH_{2,4}(\Lambda^\bullet(\rmH^1(G,\FF_p)),\FF_p(\rmH^\bullet(G,\FF_p))=0$?
\end{quest}
In Example~\ref{exmp:F2 times F2} we study the group $G=F_2\times F_2$ with trivial orientation. It satisfies the conditions of Theorem~\ref{thm:TorSES}, but $\ker d^{2,1}_2$ is non-zero. In fact, using the exact sequence we are able to determine that $\ker d_2^{2,1}\cong \FF_p$. 

The conclusions of Theorem~\ref{thm:TorSES} also allow us to relax the conditions of Theorem~\ref{thm:Positelski_Koszul} by applying the same techniques as Positselski in \cite[Theorem 4]{Positselski2005}:
\begin{mainthm}
\label{thm:WeakKoszulBPC}
Keep the notation of Theorem~\ref{thm:TorSES}. Assume $(\ker \psi^\bullet)(2)$ is a quadratic $\Lambda^\bullet(V)$-module and $H_{i,i+3}(\Lambda^\bullet(V),\ker\psi^\bullet)=0$ for all $i\in \NN_0$, then $(G,\theta)$ has the Bogomolov--Positselski property.
\end{mainthm}
This theorem again depends on the vanishing of infinitely many cohomology groups, but does not require the \lq\lq full\rq\rq\, Koszulity of $(\ker\psi^\bullet)(2)$. In Section~\ref{sec:Computation} we show that is suffices to compute three graded cohomology groups to verify the conditions of Theorem~\ref{thm:WeakKoszulBPC} for an ideal of the exterior algebra $\Lambda^\bullet(V)$.

Finally in Section~\ref{sec:ModuleKoszulityForElementaryType} we show that for a torsion-free oriented pro-$p$ group $(G,\theta)$ of \emph{elementary type} $(\ker \psi^\bullet)(2)$ is a Koszul $\Lambda^\bullet(\rmH^1(G,\FF_p))$-module and therefore not only satisfies the conditions of Theorem~\ref{thm:WeakKoszulBPC} but also the ones of Positselski's Theorem~\ref{thm:Positelski_Koszul}. Groups of elementary type are groups pro-$p$ groups constructed from Demushkin groups and free pro-$p$ groups by free pro-$p$ products and semidirect products with free abelian pro-$p$ groups. For a precise definition of this class of groups, we refer to Definition~\ref{def:ElementaryType}. Prior it was shown Quadrelli and Weigel that groups of elementary type have the Bogomolov--Positselski property (cf. \cite[Section 5]{QuadrelliWeigel2022}). 
\begin{mainthm}
\label{thm:Elementary Type satisfies Module Koszulity}
    Let $(G,\theta)$ be a torsion-free oriented pro-$p$ group of elementary type, then $(\ker \psi^\bullet)(2)$ is Koszul.
\end{mainthm}
The following conjecture is central in the study of maximal pro-$p$ Galois groups is due to I. Efrat (cf. \cites{Efrat1997, Efrat1998, Efrat1999}).
\begin{conj*}[Elementary Type Conjecture]
\label{conj:ElementaryType}
    Let $\mathbb{K}$ be a field. If $G_{\mathbb{K}}(p)$ is finitely generated, then $(G_{\mathbb{K}}(p),\theta_{\KK})$ is an oriented pro-$p$ group of elementary type.
\end{conj*}
Thus the Elementary Type Conjecture would imply together with Theorem~\ref{thm:Elementary Type satisfies Module Koszulity}, that the Module Koszulity Conjecture 1 by Positselski is valid for all fields $\KK$ satisfying \ref{ass:roots_of_unity} with $\KK^\times/\KK^{\times p}$ finite. 

The Elementary Type Conjecture is known to hold in the following cases:

\begin{itemize}
\item[(a)] $\KK$ is a local field, or an extension of transcendence degree 1 of a local field;
\item[(b)] $\KK$ is a PAC field, or an extension of relative transcendence degree 1 of a PAC field;
\item[(c)] $\KK$ is $p$-rigid {\rm (}for the definition of $p$-rigid fields see \cite{Ware1992}*{p.~722}{\rm )};
\item[(d)] $\KK$ is an algebraic extension of a global field of characteristic not $p$;
\item[(e)] $\KK$ is a valued $p$-Henselian field with residue field $\kappa$, and $G_{\kappa}(p)$ satisfies the strong $n$-Massey vanishing property for every $n>2$.
\end{itemize}

\section{(Co)homology of graded algebras}
\label{sec:CohomologyOfGradedAlgebras}
For this section we fix a base field $k$ and only consider $\ZZ$-graded $k$-vector spaces. We abbreviate the graded tensor product of graded $k$-vector spaces by $\otimes$. The purpose of this section is to recall the most basic definitions and results. For a more detailed explanation, we refer to \cite[Chapter 3]{LodayVallette2012} and \cites{PositselskiPolishchuk2005,PositselskiVishik1995,Positselski2005}.

\begin{defn}
    A \emph{graded $k$-algebra} is a graded $k$-vector space $A=\oplus_{i\in \ZZ}A_i$ together with a map $\mu:A\otimes A\to A$ of degree $0$ satisfying the usual axioms. A graded algebra $A$ is called \emph{connected} if $A_i=0$ for $i<0$ and $A_0\cong k$.
    
    A \emph{graded (right) $A$-module} $M$ over a graded algebra $A$ is a graded $k$-vector space together with a map $M\otimes A\to M$ of degree $0$ satisfying the usual identities.
\end{defn}
\begin{exmp}
    \begin{enumerate}
        \item The tensor algebra $\mathbb{T}^\bullet(V)$ over a $k$-vector space $V$ is defined by $\mathbb{T}^n(V)=V^{\otimes n}$ an product induced by:
        \begin{align*}
            (v_1\otimes ...\otimes v_j)\cdot (w_1\otimes ...\otimes w_i):=v_1\otimes ...\otimes v_j\otimes w_1\otimes ...\otimes w_i
        \end{align*}
        \item The exterior algebra $\Lambda^\bullet(V)$ over a $k$-vector space $V$ is defined by $\mathbb{T}^\bullet(V)/\langle v\otimes v : v\in V\rangle$. It is connected and graded-commutative, i.e., if $a,b\in \Lambda^\bullet(V)$ are homogeneous, then $a\cdot b=(-1)^{\deg(a)\cdot \deg(b)}b\cdot a$.
        \item The symmetric algebra $\mathbb{S}^\bullet(V)$ is similarly defined as $\mathbb{T}^\bullet(V)/\langle v\otimes w-w\otimes v:v,w\in V\rangle$ and if $n=\dim V<\infty$, then $\mathbb{S}^\bullet(V)\cong k[x_1,..,x_n]$ with the natural grading on the polynomial ring.
    \end{enumerate}
\end{exmp}
Let $A$ be a connected, graded $k$-algebra, then we denote by $A_+:=\bigoplus_{i\geq 1}A_i$ its \emph{augmentation ideal}. Consider the following complex of free $A$-modules, which is called the \emph{normalized bar-complex} of $A$ and denoted by $\mathcal{BAR}_*(A)$:
\begin{align*}
     k\leftarrow A\leftarrow  A\otimes_k A_+\leftarrow  A\otimes_k(A_+)^{\otimes 2}\leftarrow A\otimes_k(A_+)^{\otimes 3}\leftarrow ...
\end{align*}
The differentials are given by the usual formulas and are of degree $0$. It is a projective resolution of $k$, considered as trivial $A$-module, in an appropriate category of graded modules. Then for a graded $A$-module $M$ one defines the \emph{(co)homology groups of $A$ with coefficients in $M$} by
\begin{align*}
    \rmH_i(A,M)&:=\rmH_i(M\otimes_A\mathcal{BAR}_*(A))\qquad \text{and}\\
    \rmH^i(A,M)&:=\rmH^i(\Hom_{A}(\mathcal{BAR}_*(A),M)).
\end{align*}
From the grading of the normalized bar-complex and the grading of the tensor product $\bl \otimes_A\bl$ resp. $\Hom_A(\bl,\bl)$ one deduces that also the vector spaces $\rmH_i(A,M)$ and $\rmH^i(A,M)$ have a natural grading, i.e.,
\begin{align*}
    \rmH_i(A,M)=\bigoplus_{j}\rmH_{ij}(A,M)\qquad\text{and}\qquad \rmH^i(A,M)=\bigoplus_{j}\rmH^{ij}(A,M).
\end{align*}
Furthermore, $\mathcal{BAR}_*(A)$ is a DG-coalgebra, inducing a coproduct on the homology of $A$ and a product on the cohomology of $A$. Both the coproduct and coproduct respect the gradings of the (co)homology groups.
\begin{rem}
    If both $A$ and $M$ are locally finite dimensional, that is, $A_i$ and $M_i$ are finite dimensional for all $i$, then $H_{ij}(A,M)^*\cong H^{ij}(A,M)$, so the two can be used almost interchangeably (see \cite{PositselskiPolishchuk2005}*{Section 1}). 

    If $A$ (or $M$) are not locally finite dimensional, then the homology is usually better behaved than the cohomology. 
\end{rem}
For a graded module $M$, we define its \emph{$k$-shift} $M(k)$ by $M(k)_i=M_{i+k}$ for $k\in \ZZ$. Then $\rmH_{i,j}(A,M(k))\cong \rmH_{i,j+k}(A,M)$ and similarly for cohomology.

Using the normalized bar-complex it is not hard to show the following proposition:
\begin{prop}
\label{prop:CohomologyConcentrated}
    Let $A$ be a connected graded $k$-algebra and $M$ a graded $A$-module with $M_i=0$ for $i<m$ for some $m\in \ZZ$. Then for $j<i+m$ one has
    \begin{align*}
        \rmH_{ij}(A,M)=0\qquad\text{and}\qquad \rmH^{ij}(A,M)=0
    \end{align*}
\end{prop}
\subsection{Quadratic and Koszul algebras}
\label{sec:QuadraticAndKoszulAlgebras}
\begin{defn}
\label{def:quadrati algebra and module}
    A graded connected $k$-algebra $A$ is called \emph{quadratic} if the natural morphism $\mathbb{T}^\bullet(A_1)\to A$ is surjective and its kernel $J_A$ is generated by $(J_A)_2=\mathbb{T}^2(A_1)\cap J_A$ as a two-sided ideal in $\mathbb{T}^\bullet(A_1)$.

    A graded module $M$ with $M_i=0$ for $i<0$ over a graded connected $k$-algebra $A$ is called \emph{quadratic} if the natural morphism $M_0\otimes A\to M$ is surjective and its kernel $J_M$ is generated by $(J_M)_1=(M_0\otimes A_1)\cap J_M$ as an $A$-module.
\end{defn}
\begin{exmp}
    \begin{enumerate}
        \item The algebras $\mathbb{T}^\bullet(V)$, $\Lambda^\bullet(V)$, and $\mathbb{S}^\bullet(V)$ are quadratic for any $k$-vector space $V$.
        \item If $A$ is a locally finite-dimensional commutative quadratic algebra, then $A\cong k[x_1,...,x_n]/(q_i:i\in I)$, where $(q_i)_{i\in I}$ is a family of quadratic forms in the variables $x_1,...,x_n$. For example $k[x]/(x^3)$ is not quadratic.
    \end{enumerate}
\end{exmp}
\begin{con}
    Given $V$ a $k$-vector space and $R$ a subspace of $V\otimes_k V$, then one can construct a quadratic algebra  $\{V,R\}:=\mathbb{T}^\bullet(V)/(R)$ and similarly, given a connected graded algebra $A$, a $k$-vector space $H$ and $K$ a subspace of $H\otimes A_1$, then one can associate a quadratic module $\langle H,K \rangle_A:=(H\otimes A)/\langle K\rangle$.

    Using this notation, we can also construct the so called \emph{quadratic part} of an algebra resp. module. If $A$ is a connected graded algebra and $M$ a module over $A$, then
    \begin{align*}
        {\rm q}A:=\{A_1,(J_A)_2\}\qquad \text{and}\qquad {\rm q}_AM:=\langle M_0,(J_M)_1\rangle 
    \end{align*}
    using the notations from Definition~\ref{def:quadrati algebra and module}. Notice, that $A$ resp. $M$ are quadratic if and only if ${\rm q}A\cong A$ resp. ${\rm q}_AM\cong M$.
\end{con}
There is a homological criterion to determine whether a graded algebra, respectively, a module over it is quadratic:
\begin{prop}[{\cite[Chapter 1 Corollary 5.3]{PositselskiPolishchuk2005}}]
    \label{prop:QuadraticByCohomology}
    Let $A$ be a connected graded $k$-algebra and $M$ be a graded module over $A$. 
    \begin{enumerate}
        \item $M$ is quadratic if and only if $\rmH_{0,j}(A,M)=0$ for $j\neq 0$ and $\rmH_{1,j}(A,M)=0$ for $j\neq 1$.
        \item $A$ is quadratic if and only if $\rmH_{1,j}(A,k)=0$ for $j\neq 1$ and $\rmH_{2,j}(A,k)=0$ for $j\neq 2$.
    \end{enumerate}
\end{prop}
\begin{defn}
    A connected graded algebra $A$ is called \emph{Koszul} if $\rmH_{ij}(A,k)=0$ for all $i\neq j$ and a graded module $M$ over $A$ is called \emph{Koszul} if $\rmH_{ij}(A,M)=0$ for all $i\neq j$. 
\end{defn}
\subsection{Duals of locally finite-dimensional quadratic algebras}
In this section, we study a duality for quadratic algebras. We assume that all algebras are locally finite-dimensional, in order to use the isomorphism $V^*\otimes W^*\cong (V\otimes W)^*$, which doesn't hold in the infinite-dimensional context. In this case, the duality one has to consider is between algebras and coalgebras (see for example \cites{PositselskiVishik1995,Positselski2005}).

\begin{defn}
    Let $V$ and $H$ be finite dimensional $k$-vector spaces and $R\subseteq V\otimes V$ and $K\subseteq H\otimes V$ subspaces. Then we define the \emph{quadratic duals}
    \begin{align*}
        \{V,R\}^!:=\{V^*,R^\perp\}\qquad \text{and}\qquad \langle H,K\rangle^!_{\{V,R\}}:=\langle H^*,K^\perp\rangle_{\{V,R\}^!}.
    \end{align*}
    Here $R^\perp$ is the orthogonal complement of $R$ with respect to the pairing $(V\otimes V)\times (V^*\otimes V^*)\to k$ defined by $(v\otimes w,f\otimes g)\mapsto f(v)g(w)$. Similarly, $K^\perp$ is the orthogonal complement of $K$ with respect to a similar pairing $(H\otimes V)\times (H^*\otimes V^*)\to k$.
\end{defn}
If $M$ is a quadratic module over a quadratic algebra $A$, then we sometimes simply write $M^!$ instead of $M^!_A$, if the algebra $A$ is clear from the context.
\begin{exmp}
    For $V$ a finite dimensional $k$-vector space one has $\mathbb{T}^\bullet(V)^!\cong k$ and $\Lambda^\bullet(V)^!\cong \mathbb{S}^\bullet(V^*)$. The quadratic dual of a trivial module over a quadratic algebra is free over the dual of the algebra. 
\end{exmp}
The quadratic dual of an algebra and its modules appears naturally, when studying the \lq\lq diagonal cohomology\rq\rq. The following Proposition is due to Priddy \cite{Priddy1970} and Löfwall \cite{Loeffwall1986} and can be found in \cite{PositselskiPolishchuk2005}*{Chapter 1 Proposition 3.1}.
\begin{prop}
\label{prop:CohomologyAndDuals}
    Let $A$ be a connected graded algebra and $M$ a graded $A$-module with $M_i=0$ for $i<0$. Then 
    \begin{enumerate}
        \item $\bigoplus_{i} \rmH^{i,i}(A,k)\cong (qA)^!$ as graded algebras.
        \item $\bigoplus_{i}H^{i,i}(A,M)\cong ({\rm q}_{A}M)^!$ as graded $(qA)^!$-modules.
    \end{enumerate}
\end{prop}
\begin{prop}[\cite{PositselskiPolishchuk2005}*{Chapter 2, Cor. 3.3 and Cor. 3.5 (M)}]
    \label{prop:SimultaniousKoszul}
    Let $A$ be a quadratic algebra, then $A$ is Koszul if and only if its quadratic dual $A^!$ is Koszul.  
    
    Assume that $A$ is Koszul and $M$ is a quadratic $A$-module, then $M$ is Koszul over $A$ if and only if $M_A^!$ is Koszul over $A^!$. More precisely, for $a,b\in \NN_0$ the following are equivalent:
    \begin{enumerate}
        \item $\rmH^{ij}(A,M)=0$ for $i-1\leq a$ and $0<j-i\leq b$;
        \item $\rmH^{ij}(A^!,M_A^!)=0$ for $i-1\leq b$ and $0<j-i\leq a$.
    \end{enumerate}  
\end{prop}
The following construction allows us to produce new algebras and modules from known ones and will prove useful in Section~\ref{sec:ModuleKoszulityForElementaryType}. It is spelled out in more detail in \cite{PositselskiPolishchuk2005}*{Chapter 3 §1}.
\begin{con}
\label{con:TensorProductOfAlgebras}
    Let $A$ and $B$ be connected, graded $k$-algebras. Then we define $A\otimes^{-1}B$ to be isomorphic to the graded tensor product $A\otimes B$ as $k$-vector space together with the product given on homogeneous elements $a_1,a_2\in A$ and $b_1,b_2\in B$ by
    \begin{align*}
        (a_1\otimes^{-1} b_1)\cdot (a_2\otimes^{-1} b_2)=(-1)^{\deg(b_1)\deg(a_2)}(a_1a_2\otimes^{-1}b_1b_2).
    \end{align*}
    For example, if $V$ and $W$ are $k$-vector spaces, then $\Lambda^\bullet(V\oplus W)\cong \Lambda^\bullet(V)\otimes^{-1}\Lambda^\bullet(W)$. If $M$ is a graded $A$-module and $N$ a graded $B$-module, then one can define a graded $A\otimes^{-1}B$-module $N\otimes^{-1}M$ by similar formulas.
\end{con}
\begin{prop}[\cite{PositselskiPolishchuk2005}*{Chapter 3 Prop. 1.1 and Cor. 1.2}]
\label{prop:CohomologyOfTensorProduct}
    Let $A$ and $B$ be connected, graded $k$-algebras. Then $A\otimes^{-1}B$ is Koszul if and only if $A$ and $B$ are Koszul. For a graded $A$-module $M$ and a graded $B$-module $N$ we have
    \begin{align*}
        \rmH_{n}(A\otimes^{-1}B,M\otimes^{-1}N)\cong \bigoplus_{i+j=n}\rmH_i(A,M)\otimes \rmH_j(B,N)
    \end{align*}
    as graded $k$-vector spaces.
\end{prop}

\section{Quadratic Galois cohomology algebras and Koszulity conjectures}
\label{sec:GaloisCohomology}
Let $\mathbb{K}$ be a field and denote by $\mu_p$ the group of $p^{\text{th}}$ roots of unity in a fixed separable closure $\mathbb{K}^{s}$. Let $G_\mathbb{K}=\Gal(\mathbb{K}^{s}/\mathbb{K})$ be the absolute Galois group of $\mathbb{K}$. We denote by $K^M_n(\mathbb{K})$ the $n$-th Milnor $K$-group of $\mathbb{K}$, which is defined as
\begin{align*}
    K^M_n(\mathbb{K})=(\mathbb{K}^\times)^{\otimes n}/\langle a_1\otimes ....\otimes a_n:a_i+a_j=1\text{ for some }i\neq j\rangle.
\end{align*}
Then $K^M_\bullet(\mathbb{K})$ with the canonical product is a graded ring and $K^M_\bullet(\mathbb{K})\otimes \FF_p$ is a quadratic $\FF_p$-algebra. In \cite{Tate1976} Tate showed the existence of an algebra homomorphism $h_p:K^M_\bullet(\mathbb{K})\otimes \FF_p\to \bigoplus_{i}\rmH^i(G_\mathbb{K},\mu_p^{\otimes i})$, extending the Kummer isomorphism in degree $1$.

The following theorem was proven by Rost and Voevodsky together with a 
``patch'' by Weibel (cf. \cites{Voevodsky2011,Weibel2008,Weibel2009} and resolved a conjecture by Bloch and Kato.
\begin{thm}[Norm residue isomorphism theorem]
\label{thm:BlochKato}
    The map $h_p$ above is an isomorphism of graded algebras. In particular, the algebra $\bigoplus_{i}\rmH^i(G_\mathbb{K},\mu_p^{\otimes i})$ is quadratic.
\end{thm}
If $\mathbb{K}$ contains a primitive $p^{\text{th}}$ root of unity, then $\mu_p^{\otimes n}\cong \FF_p$ (non-canonically) and thus the algebra $\rmH^\bullet(G_\mathbb{K},\FF_p)$ is quadratic in this case. Furthermore, the $\FF_p$-cohomology algebras of $G_\mathbb{K}$ and its maximal pro-$p$ quotient $G_\mathbb{K}(p)$ agree.

Positselski showed in \cite{PositselskiVishik1995}, that Theorem~\ref{thm:BlochKato} would follow from the Koszulity of $K^M_\bullet(\mathbb{K})\otimes \FF_p$ if $G_\mathbb{K}$ is a pro-$p$ group. He posed the following conjectures in \cite{Positselski2014}, which were suggested by his previous work:
\begin{conj}[Koszulity Conjecture]
\label{conj:KoszulityConjecture}
    For any field $\mathbb{K}$ containing a primitive root $p^{\text{th}}$ root of unity, the algebra $K^M_\bullet(\mathbb{K})\otimes \FF_p$ is Koszul.
\end{conj}
\begin{conj}[Module Koszulity Conjecture 1]
\label{conj:ModuleKoszulityConjecture1}
    Let $\mathbb{K}$ be a field satisfying \ref{ass:roots_of_unity}. Define  $J_\mathbb{K}$ to be the kernel of the natural map $ \Lambda^\bullet(\mathbb{K}^\times/(\mathbb{K}^{\times p}))\to K^M_\bullet(\mathbb{K})\otimes \FF_p$, then $J_\mathbb{K}(2)$ is a Koszul module over $\Lambda^\bullet(\mathbb{K}^\times/(\mathbb{K}^{\times p}))$.
\end{conj}
\begin{rem}
    The Module Koszulity Conjecture~1 implies the Koszulity Conjecture by a simple argument using a change of rings spectral sequence.
    
    Theorem 2 of \cite{Positselski2005} shows that the Module Koszulity Conjecture~1 implies the Bogomolov--Positselski Conjecture. 
\end{rem}
These conjectures are known to hold for some classes of fields (e.g. number fields, local fields (cf. \cite{Positselski2014}). Recently, Miná{\^c}, Pasini, Quadrelli, and Tân made some progress on the first of the above conjectures by showing that for oriented pro-$p$ groups $(G,\theta)$ of elementary type, the algebra $\rmH^\bullet(G,\FF_p)$ has the PBW property and is therefore Koszul (see \cite{MinacPasiniQuadrelliTan2022}). In \cite{MinacPasiniQuadrelliTan2021} they furthermore proved that if the maximal pro-$p$ quotient of $G_\mathbb{K}$ is a mild pro-$p$ group, then Conjecture \ref{conj:KoszulityConjecture} is true. In 2020 Snopce and Zalesskii proved that the cohomology algebra of a right-angled Artin pro-$p$ group is universally Koszul if and only if it is the maximal pro-$p$ Galois group of a field $\mathbb{K}$ containing a primitive $p^{\text{th}}$ root of unity (see \cite{SnopceZalesskii2022}).

\section{The Bogomolov--Positselski conjecture and the proofs of Theorem~\ref{thm:TorSES} and Theorem~\ref{thm:WeakKoszulBPC}}
\subsection{Proof of Theorem~\ref{thm:TorSES}}
We start with a general proposition about the homology of graded algebras and then apply it to the group-theoretic situation.
\begin{prop}
    \label{prop:HomologicalAlgebra}
    Let $A$ be a connected graded $k$-algebra and 
    \begin{equation*}
        0\to K\to M\overset{\varphi}{\to}A\overset{\pi}{\to}B\to 0
    \end{equation*}
     be an exact sequence of $A$-modules with degree preserving homomorphisms. Assume the following two conditions:
    \begin{enumerate}
        \item \label{ass:Mconcentrated} $M_i=0$ for $i< 1$ (thus $\pi$ is an isomorphism in degree $0$ and $1$);
        \item \label{ass:Kconcentrated} $K_i=0$ for $i<4$ (this implies with (\ref{ass:Mconcentrated}) that $\varphi$ is injective in degree $2$ and $3$);
    \end{enumerate}
    Then there is an exact sequence
    \begin{equation*}
            0\to H_{1,4}(A,M) \to H_{2,4}(A,B)\to K_4\to H_{0,4}(A,M)\to H_{1,4}(A,B)\to 0
    \end{equation*}
\end{prop}
\begin{proof}
    Consider the acyclic complex $C_*:=[0\leftarrow B\leftarrow A \leftarrow M\leftarrow K\leftarrow 0]$ (we choose $B$ to be in degree $0$, but it does not affect our arguments). Now, since the category of graded modules with degree preserving homomorphisms has enough projectives, there exist (projective) Cartan-Eilenberg resolutions, there is a homological spectral sequence $D^1_{s,t}:=H_{t,4}(A,C_s)\Rightarrow 0$. Since $D^1$ is concentrated in $4$ columns, we conclude $D^4_{s,t}=0$ for all $s$ and $t$. We denote the differentials by $\partial^r_{s,t}$.

    As $A$ is a free $A$ module, we have ${D}{^1_{1,t}}=0$ for all $t$ and by assumption (\ref{ass:Kconcentrated}) $\rmH_{0,4}(A,K)\cong K_4$. A variant of Proposition~\ref{prop:CohomologyConcentrated} implies $H_{i,4}(A,K)=0$ for $i\geq 1$ and similarly $H_{i,4}(A,M)=0$ for $i\geq 3$ (by assumption (\ref{ass:Mconcentrated})). Additionally $H_0(A,B)=B\otimes_Ak=k$, which is concentrated in degree $0$. Thus, we see that the first page of the spectral sequence can be described as depicted in Figure~\ref{fig:FirstPageOfD}.

    \begin{figure}[ht]
        \centering
        \begin{tikzpicture}
            \matrix (m) [matrix of math nodes,
                nodes={text width= 2cm,align=center, minimum height=5ex,outer sep=-5pt},
                column sep=.2cm,row sep=1ex,
                column 1/.style={nodes={text width=.6cm}}
                ]{
                    3     &   H_{3,4}(A,B)    &   0    &    0    &  0  \\
                    2     &   H_{2,4}(A,B)    &   0  &   H_{2,4}(A,M)     &  0  \\
                    1     &   H_{1,4}(A,B)    &   0   &    H_{1,4}(A,M)    &   0 \\
                    0     &   0   &    0    &   H_{0,4}(A,M)     &  K_4 \\
                \quad\strut &   0  &  1  &  2  & 3   \\};
            
                
                \draw[thick] (m-1-1.north east) -- (m-\the\pgfmatrixcurrentrow-1.south east) ;
                \draw[thick] (m-\the\pgfmatrixcurrentrow-1.north west) -- (m-\the\pgfmatrixcurrentrow-\the\pgfmatrixcurrentcolumn.north east) ;
        \end{tikzpicture}
        \caption{First page of the spectral sequence $D^1_{*,*}$}
        \label{fig:FirstPageOfD}
    \end{figure}
    
The only non-zero differential on the first page is is ${\partial}{^1_{3,0}}:K_4\to \rmH_{0,4}(A,M)$. The resulting second page is shown in Figure~\ref{fig:SecondPageOfD}.

    \begin{figure}[ht]
        \centering
        \begin{tikzpicture}
            \matrix (m) [matrix of math nodes,
                nodes={text width= 2cm,align=center, minimum height=5ex,outer sep=-5pt},
                column sep=.2cm,row sep=1ex,
                column 1/.style={nodes={text width=.6cm}}
                ]{
                    3     &   H_{3,4}(A,B)    &       &    0    &  0  \\
                    2     &   H_{2,4}(A,B)    &     &   H_{2,4}(A,M)     &  0  \\
                    1     &   H_{1,4}(A,B)    &      &    H_{1,4}(A,M)    &   0 \\
                    0     &   0   &    0    &   \coker({\partial}{^1_{3,0}})    &  \ker({\partial}{^1_{3,0}})  \\
                \quad\strut &   0  &  1  &  2  & 3   \\};
            
                
                \draw[thick] (m-1-1.north east) -- (m-\the\pgfmatrixcurrentrow-1.south east) ;
                \draw[thick] (m-\the\pgfmatrixcurrentrow-1.north west) -- (m-\the\pgfmatrixcurrentrow-\the\pgfmatrixcurrentcolumn.north east) ;
                \draw[->] (m-4-4.west) -- node[above] {${\partial}{^2_{2,0}}$} (m-3-2.east);
                \draw[->] (m-3-4.west) -- node[above] {${\partial}{^2_{2,1}}$} (m-2-2.east);
                \draw[->] (m-2-4.west) -- node[above] {${\partial}{^2_{2,2}}$} (m-1-2.east);
        \end{tikzpicture}
        \caption{Second page of the spectral sequence $D^2_{*,*}$}
        \label{fig:SecondPageOfD}
    \end{figure}

We conclude that all the maps ${\partial}{^2_{2,t}}$ are injective as $\ker ({\partial}{^2_{2,t}})={D}{^3_{2,t}}={D}{^\infty_{2,t}}$. Moreover, ${\partial}{^2_{2,0}}$ and ${\partial}{^2_{2,2}}$ have to be isomorphisms. Thus, the third page has only two non-zero entries is described in Figure~\ref{fig:ThirdPageOfD}.

    \begin{figure}[ht]
    \centering
        \begin{tikzpicture}
            \matrix (m) [matrix of math nodes,
                nodes={text width= 2cm,align=center, minimum height=5ex,outer sep=-5pt},
                column sep=.2cm,row sep=1ex,
                column 1/.style={nodes={text width=.6cm}}
                ]{
                    3     &   0    &   0    &    0    &  0  \\
                    2     &   \coker{\partial}{^2_{2,1}}   &   0  &   0     &  0  \\
                    1     &   0    &   0   &    0    &   0 \\
                    0     &   0   &    0    &   0    &  \ker({\partial}{^1_{3,0}})  \\
                \quad\strut &   0  &  1  &  2  & 3   \\
                };
            
            
            \draw[thick] (m-1-1.north east) -- (m-\the\pgfmatrixcurrentrow-1.south east) ;
            \draw[thick] (m-\the\pgfmatrixcurrentrow-1.north west) -- (m-\the\pgfmatrixcurrentrow-\the\pgfmatrixcurrentcolumn.north east) ;
            \draw[->] (m-4-5.west) -- node[above] {${\partial}{^3_{3,0}}$} (m-2-2.east);
        \end{tikzpicture}
        \caption{Second page of the spectral sequence $D^2_{*,*}$}
        \label{fig:ThirdPageOfD}
    \end{figure}
Similarly to the discussion before, the differential ${\partial}{^3_{3,0}}$ has to be an isomorphism. Thus we arrive at two short exact sequences:
\begin{equation*}
    \begin{tikzcd}[row sep=small]
        0\arrow[r]&H_{1,4}(A,M)\arrow[r,"{{\partial}{^2_{2,1}}}"] &H_{2,4}(A,B)\arrow[r]&\ker({\partial}{^1_{3,0}})\arrow[r]&0\\
        0\arrow[r]&\im({\partial}{^1_{3,0}})\arrow[r] &H_{0,4}(A,M)\arrow[r]&H_{1,4}(A,B)\arrow[r]&0\\
    \end{tikzcd}
\end{equation*}
Splicing these sequences together yields the desired $5$-term exact sequence.
\end{proof}
Now we are ready to prove 
\begin{repthm}{thm:TorSES}
    Let $(G,\theta)$ be a torsion-free, Kummerian, oriented pro-$p$ group, such that $G$ is $\rmH^\bullet$-quadratic. We set  $V:=\rmH^1(G,\FF_p)$, $B:=\rmH^\bullet(G,\FF_p)$ and let $N$ be the graded $\Lambda^\bullet(V)$-module $\rmH^\bullet(G(\theta),\rmH^1(I_\theta,\FF_p))$. Then there is an exact sequence:
    \begin{equation*}
        \begin{tikzcd}[column sep=small]
            0\arrow[r]&\rmH_{1,2}(\Lambda^\bullet(V),N)\arrow[r] &\rmH_{2,4}(\Lambda^\bullet(V),B)\arrow[r] &\ker d_2^{2,1}\arrow[r]&\rmH_{0,2}(\Lambda^\bullet(V),N)\arrow[r]&0
        \end{tikzcd}
    \end{equation*}
    In particular $\ker d_2^{2,1}=0$ if and only if the first map is an isomorphism and $H_{0,2}(\Lambda^\bullet(V),N)=0$.
\end{repthm}
\begin{proof}
    Consider the Hochschild-Serre spectral sequence from (\ref{equ:HochschildSerreSpectral}), which is multiplicative. Then $E_2^{\bullet,0}=\rmH^\bullet(G(\theta),\FF_p)\cong \Lambda^\bullet(V)$ as graded algebra, and $N=E_2^{\bullet,1}$ is a $\Lambda^\bullet(V)$-module. By \cite[Proposition 4.4 (ii)]{QuadrelliWeigel2022}, we have an exact sequence
    \begin{equation*}    
            0\longrightarrow \ker d^{\bullet,1}_2\longrightarrow N\overset{d^{\bullet,1}_2}\longrightarrow A\overset{\psi^\bullet}{\longrightarrow}B\longrightarrow 0
    \end{equation*}
    Since $d^{\bullet,1}_2$ is of degree $2$, we just replace $N$ by $M:=N(-2)$, which immediately implies $M_i=0$ for $i\leq 1$. Similarly, we set $K:=(\ker d_2^{\bullet,1})(-2)$ and get $K_i=0$ for $i\leq 1$. To show $K_2=0$, as the injectivity of $d^{0,1}_2$ follows directly from the $5$-term sequence associated to the spectral sequence. For $K_3$ we consider the following exact sequence
    \begin{equation*}
        \rmH^2(G(\theta),\FF_p)\overset{\psi^2}\to \ker(\rmH^2(G,\FF_p)\to E_2^{0,2})\to E^{1,1}_2\overset{d^{1,1}_2}\to  H^3(G(\theta),\FF_p) 
    \end{equation*}
    coming from the seven term sequence associated to the spectral sequence. Since $\psi^2$ is surjective onto $\rmH^2(G,\FF_p)$, the differential $d^{1,1}_2$ has to be injective and $K_2=K_3=0$. Now we can apply Proposition~\ref{prop:HomologicalAlgebra} and get the following exact sequence:
    \begin{align*}
            0\to\rmH_{1,2}(\Lambda^\bullet(V),N)&\to\rmH_{2,4}(\Lambda^\bullet(V),B)\to\ker d_2^{2,1}\to ...\\...&\to\rmH_{0,2}(\Lambda^\bullet(V),N)\to\rmH_{1,4}(\Lambda^\bullet(V),B)\to 0
    \end{align*}
    It remains to show that $\rmH_{1,4}(\Lambda^\bullet(V),B)$. By \cite[Section 4.2]{QuadrelliWeigel2022} the $\ker\psi^\bullet$ is generated in degree $2$. Thus we get $\rmH_{1,4}(\Lambda^\bullet(V),B)\cong \rmH_{0,4}(\Lambda^\bullet(V),\ker \psi^\bullet)=0$, yielding the desired exact sequence.
\end{proof}
Theorem~\ref{thm:TorSES} implies, that if $G$ has the Bogomolov--Positselski property, then 
\begin{align*}
    0=H_{0,2}(\Lambda^\bullet(V),N)=(\FF_p\otimes_{\Lambda^\bullet(V)}N)_2=N_2/(\Lambda^1(V)\cdot  N_1+\Lambda^2(V)\cdot N_0).
\end{align*}
Using that $\rmH_{0,1}(\Lambda^\bullet(V),N)=0$ one can deduce the following corollary:
\begin{cor}
\label{cor:CupProductSurjective}
    Let $(G,\theta)$ be as in Theorem~\ref{thm:TorSES} and assume that it has the Bogomolov--Positselski property, then the map
    \begin{align}
        \label{eq:CupProductCondition}
        \rmH^1(G(\theta),\rmH^1(K_\theta(G),\FF_p))\otimes \rmH^1(G(\theta),\FF_p)\to \rmH^2(G(\theta),\rmH^1(K_\theta(G),\FF_p))
    \end{align}
    which is induced by the cup product is surjective. The converse implication holds if
    \begin{align*}
        \rmH_{2,4}(\Lambda^\bullet(V),B)=0 .
    \end{align*}
\end{cor}
\begin{rem}
    If $(G,\theta)=(G_\mathbb{K}(p),\theta_{\KK})$ for a field $\KK$, then the cup product can be written in terms of the field $\mathbb{L}:=\sqrt[p^\infty]{\mathbb{K}}$. In particular the cup product in (\ref{eq:CupProductCondition}) becomes
    \begin{align*}
        \rmH^1(G(\theta),\mathbb{L}^\times\otimes \FF_p)\otimes \rmH^1(G(\theta),\FF_p)\to \rmH^2(G(\theta),\mathbb{L}^\times \otimes \FF_p).
    \end{align*}
\end{rem}
\begin{exmp}
\label{exmp:F2 times F2}
    Denote by $F_2$ the free pro-$p$ group on two generators. Then consider the group $G=F_2\times F_2$ as oriented pro-$p$ group with trivial orientation. Then it is Kummerian, as $G^{\rm ab}\cong \ZZ_p^2\times \ZZ_p^2$ is torsion-free, and $G$ has quadratic cohomology. Thus Theorem~\ref{thm:TorSES} is applicable to $G$. Since $G'=F_2'\times F_2'$ is not a free pro-$p$ group, we get the $\ker d_2^{2,1}\neq 0$ by \cite{QuadrelliWeigel2022}*{Theorem 4.5}. Note that by \cite{Quadrelli2014}*{Theorem 5.6} this group is not the maximal pro-$p$ group of a field containing a primitive $p^{\rm th}$ root of unity. 

    If we suppress the coefficients in cohomology, we implicitly take coefficients in $\FF_p$. By the Künneth formula we have $\rmH^\bullet(G)\cong \rmH^\bullet(F_2)\otimes^{-1}\rmH^\bullet(F_2)$ and it is not hard to see that 
    \begin{align*}
        \rmH_{i,j}(\Lambda^\bullet(\FF_p^2),\rmH^\bullet(F_2))\cong \begin{cases}
            \FF_p&\text{if }i=j=0,\\
            \mathbb{S}^{i-1}(\FF_p^2)&\text{if }0<i=j-1\text{ and}\\
            0&\text{otherwise.}
        \end{cases}
    \end{align*}
    By Construction~\ref{con:TensorProductOfAlgebras} we conclude, that $\rmH_{2,4}(\Lambda^\bullet(\rmH^1(G)),\rmH^\bullet(G))\cong \FF_p$. We now compute $\rmH^\bullet(G^{\rm ab},\rmH^1(G'))$. By the Künneth formula again, we have that $\rmH^1(G')\cong \rmH^1(F_2')\oplus \rmH^1(F_2')$. One can also see quite easily from the spectral sequence associated to $1\to F_2'\to F_2\to \ZZ_p^2\to 1$, that
    \begin{align*}
        \rmH^0(\ZZ_p^2,\rmH^1(F_2'))\cong \FF_p\quad \text{and}\quad \rmH^n(\ZZ_p^2,\rmH^1(F_2'))=0\text{ for }n>0.
    \end{align*}
    Thus --- again by the Künneth formula --- we have
    \begin{align*}
        \rmH^\bullet(G^{\rm ab},\rmH^1(G'))\cong \Lambda^\bullet(\FF_p^2)\otimes \FF_p\oplus \FF_p\otimes \Lambda^\bullet(\FF_p^2)
    \end{align*}
    and we conclude, that $\rmH^\bullet(G^{\rm ab},\rmH^1(G'))$ is a Koszul $\Lambda^\bullet(\rmH^1(G))$-module, showing by Theorem~\ref{thm:TorSES}, that $\ker d^{2,1}_2\cong \FF_p$.
\end{exmp}

\subsection{Proof of Theorem~\ref{thm:WeakKoszulBPC}}
The central theorem required for the proof of \cite{Positselski2005}*{Theorem 2} is \cite{Positselski2005}*{Theorem 4}. We adapt this theorem and combine it with Theorem~\ref{thm:TorSES} to weaken the required conditions. 
\begin{rem}
\label{rem:Conilpotent group algebra}
    Theorem~2 of \cite{Positselski2005} is formulated in the language of coalgebras, which we have not introduced in this paper.

    For the notion of a conilpotent coalgebra and its cohomology, we refer to \cite{Positselski2005}*{Section 4}. A typical example of a conilpotent coalgebra is the completed group coalgebra $\FF_p(\!(G)\!):=\varinjlim_{N\unlhd_oG} \FF_p[G/N]^*$ for a pro-$p$ group $G$. Furthermore, in that situation, any discrete $p$-torsion $G$-module $M$ can be considered as a comodule over $\FF_p(\!(G)\!)$ and their cohomology agrees:
    \begin{align*}
        \rmH^i(\FF_p(\!(G)\!),M)\cong \rmH^i(G,M).
    \end{align*}
    In fact, for the proof of Theorem~\ref{thm:WeakKoszulBPC} it is sufficient to consider this special case in Proposition~\ref{thm:conilpotent-coalgebras}.
\end{rem}

\begin{prop}
\label{thm:conilpotent-coalgebras}
    Let $C$ be a conilpotent coalgebra over a field $k$, such that its cohomology algebra $A:=\rmH^{\bullet}(C,k)$ is Koszul, and $P$ be a comodule over $C$. Consider the graded $A$-module $M:=\rmH^\bullet(C,P)$. Assume that
    \begin{enumerate}
        \item \label{ass:TorVanishingCondition} the quadratic $A$-module $q_AM$ satisfies $\rmH_{i,i+1}(A,q_AM)=0$ for all $i\in \NN_0$;
        \item \label{ass:IsoAndMonoCondition}the natural morphism of graded $A$-modules $q_AM\to M$ is an isomorphism in degree $1$ and a monomorphism in degree $2$.
    \end{enumerate}
    Then the comparison map $q_AM\to M$ is an isomorphism in degree $2$. In particular $\rmH_{0,2}(A,M)=0$.
\end{prop}
\begin{proof}
    The proof follows verbatim the one of \cite{Positselski2005}*{Theorem 2} by Positselski. The only difference is that in the last paragraphs the comodule ${\rm q}_{\gr_N C}(\gr_N P)$ is not Koszul, but an analog of Proposition~\ref{prop:SimultaniousKoszul} shows that $\rmH_{2}(\gr_N C,\gr_NP)$ is concentrated in degree $2$, which is --- by the discussion in the last two paragraphs of the proof --- sufficient to conclude the desired property.
\end{proof}

\begin{repthm}{thm:WeakKoszulBPC}
    Keep the notation of Theorem~\ref{thm:TorSES}. Assume $(\ker \psi^\bullet)(2)$ is a quadratic $A$-module and $H_{i,i+3}(\Lambda^\bullet(V),\ker\psi^\bullet)=0$ for all $i\in \NN_0$, then $(G,\theta)$ has the Bogomolov--Positselski property.
\end{repthm}

\begin{proof}
Following Remark~\ref{rem:Conilpotent group algebra}), we set $C:=\FF_p(\!(G(\theta))\!)$, and consider the discrete $G(\theta)$-module $P:=\rmH^1(K_\theta(G),\FF_p)$ as a comodule over $C$. First of all, $H^\bullet(C,\FF_p)\cong H^\bullet(G(\theta),\FF_p)\cong \Lambda^\bullet(V)$ is  Koszul.

Now set $M:=\rmH^\bullet(\FF_p(\!(G(\theta))\!),P)$. The arguments of the proof of Theorem~\ref{thm:TorSES} yield that $(\ker \psi^\bullet)_2\cong M_0$ and $(\ker \psi^\bullet)_3\cong M_1$. Since we assumed that $(\ker \psi^\bullet)(2)$ is quadratic, we have $(\ker \psi^\bullet)(2)\cong q_AM$. Hence, condition (\ref{ass:TorVanishingCondition}) of Proposition~\ref{thm:conilpotent-coalgebras} is satisfied. For condition (\ref{ass:IsoAndMonoCondition}), we notice that the composition $(\ker \psi^\bullet)_4\cong ({\rm q}_AM)_2\to M_2\to (\ker \inf^\bullet)_4$ is the identity and therefore ${\rm q}_AM\to M_2$ is a monomorphism.

Thus, Proposition~\ref{thm:conilpotent-coalgebras} yields $\rmH_{0,2}(\Lambda^\bullet(V),M)=0$. By the assumption that $(\ker \psi^\bullet)(2)$ is quadratic, we get that  
\begin{align*}
    H_{2,4}(\Lambda^\bullet(V),B)\cong H_{1,4}(\Lambda^\bullet(V),\ker \psi^\bullet)=0
\end{align*}
By Theorem~\ref{thm:TorSES} we get the Bogomolov--Positselski property.
\end{proof}
Note that the condition, that $(\ker \psi^\bullet)(2)$ is quadratic is very natural and expected in the Galois theoretic context, but not satisfied automatically for any ideal of $\Lambda^\bullet(V)$ generated in degree $2$. The following counter example is due to Simone Blumer and was privately communicated to the author.
\begin{exmp}
    \label{exmp:IdealNotQuadratic}
    Choose $V$ to be a vector space with basis $x,y,u,v$ over a field and $I$ to be the two-sided ideal of $\Lambda^\bullet(V)$ generated by $x\wedge y+u\wedge v$, then for any $0\neq t\in V$, we have $t\wedge(x\wedge y+u\wedge v)\neq 0$, so the quadratic part of $I$ would be free of rank $1$, but $I$ is not free as $x\wedge u\wedge (x\wedge y+u\wedge v)=0$.
\end{exmp}
\subsection{Theorem~\ref{thm:WeakKoszulBPC} depends on only three graded cohomology groups}
\label{sec:Computation}
The goal of this section is to show that by dualizing in an appropriate way, it is sufficient to compute three graded cohomology groups to verify the conditions of Theorem~\ref{thm:WeakKoszulBPC} for a homogeneous ideal $I$ of $\Lambda^\bullet(V)$ with $I_0=I_1=0$.

We will need the following small lemma:
\begin{lem}
\label{lem:monomorphismDual}
    Let $A$ be a quadratic $k$-algebra and $f:M\to N$ a monomorphism of quadratic $A$-modules, then the (quadratic) dual map $f^!:N^!_A\to M^!_A$ is an epimorphism of quadratic modules and its kernel is generated in degree $0$ by $\coker(f_1)^*$.
\end{lem}
\begin{proof}
    If we consider the long exact sequence of $\rmH^\ast(A,\bl)$ associated to the exact sequence $0\to M\to N\to \coker f\to 0$, we see that the following sequence is exact for every $i$ by Proposition~\ref{prop:CohomologyConcentrated} and \ref{prop:CohomologyAndDuals}.
    \begin{equation*}
        \begin{tikzcd}[row sep=tiny]
            \rmH^{i+1,i}(A,\coker f)\arrow[d,equal] &\arrow[l]\rmH^{i,i}(A,M)\arrow[d,equal]&\arrow[l]\rmH^{i,i}(A,N)\arrow[d,equal]\\
            0&\arrow[l](M^!_A)_i&\arrow[l,swap,"f^!_i"](N^!_A)_i
        \end{tikzcd}
    \end{equation*}
    Therefore, $f^!$ is an epimorphism. To see that $\ker f^!$ is generated in degree $0$, we use the long exact sequence for $\rmH^\ast({A^!},\bl)$ to get for any $i\neq 0$
    \begin{equation*}
        \begin{tikzcd}
            0\arrow[r]&\rmH^{0,i}({A^!},\ker f^!)\arrow[r]&\rmH^{1,i}({A^!},M^!_A)\arrow[r,"\alpha_i"]&\rmH^{1,i}({A^!},N^!_A)
        \end{tikzcd}
    \end{equation*}
    Since $N_A^!$ and $M_A^!$ are quadratic, we get $\rmH^{0,i}(A^!,\ker f^!)=0$ for $i\neq 0,1$. For $i=1$, we have $ \rmH^{1,1}(A,N_A^!)\cong N_1 $ and $\rmH^{1,1}(A,M_A^!)\cong M_1$. Furthermore, the map $\alpha_1$ agrees with $f_1$ and is injective. Therefore, $\rmH^{0,1}(A,\ker f^!)=0$ and $\ker f^!$ is generated in degree $0$. It is not difficult to derive the equality $\rmH^{0,0}(A,\ker f^!)\cong \coker(f_1)^*$ in a similar way.
\end{proof}
Let $V$ be a finite-dimensional $k$-vector space and set
\begin{align*}
    J:=\ker (\mathbb{S}^{\bullet-1}(V^*)\otimes V^*\to \mathbb{S}^{\bullet}(V^*)).
\end{align*}
Notice that $J_i=0$ if $i\leq 1$ and $J_2\subseteq V^\ast\otimes V^*$.
\begin{prop}
\label{prop:ThreeCohomologyGroupsAreEnough}
    Let $I$ be an ideal of $\Lambda^\bullet(V)$ such that $I_0=I_1=0$ and $I(2)$ is a quadratic $\Lambda^\bullet(V)$-module, then $\rmH^{i,i+3}(\Lambda^\bullet(V),I)=0$ if and only if
    \begin{align*}
        \rmH^{2}(\mathbb{S}^\bullet(V^*),J/\langle W^*\rangle)
    \end{align*}
    is concentrated in degree $4$, where $W:=\Lambda^2(V)/I_2$ is interpreted as a graded vector space concentrated in degree $2$. 
\end{prop}
\begin{rem}
    This shows, that three graded cohomology groups are sufficient to determine, whether the conditions of Theorem~\ref{thm:WeakKoszulBPC} are satisfied for an ideal $I$ of $\Lambda^2(V)$, namely
    \begin{align*}
        \rmH^0(\Lambda^\bullet(V),I),\quad \rmH^1(\Lambda^\bullet(V),I),\quad\text{ and}\quad \rmH^2(\mathbb{S}^\bullet(V^*),J/\langle W^*\rangle).
    \end{align*}
\end{rem}
\begin{proof}[Proof of \ref{prop:ThreeCohomologyGroupsAreEnough}]
    When considering quadratic duals, we suppress the respective algebra in the notation, as the dual is always intended with respect to $\Lambda^\bullet(V)$.
    
    We first of all show, that $J(2)$ is the quadratic dual of the $\Lambda^\bullet(V)$-module $L_2(2)$, where $L_k$ is defined by $(L_k)_i=\Lambda^i(V)$ if $i\geq k$ and $L_i=0$ otherwise. By \cite{PositselskiPolishchuk2005}*{Chapter 2, Prop. 1.1} the modules $L_k(k)$ are Koszul. The short exact sequence $0\to L_1\to \Lambda^\bullet(V)\to k\to 0$ shows that
    \begin{align*}
        \rmH^{i,i+1}(\Lambda^\bullet(V),L_1)\cong \rmH^{i+1,i+1}(\Lambda^\bullet(V),k)=\mathbb{S}^{i+1}(V^*).
    \end{align*}
    By the long exact sequence associated to $0\to L_2\to L_1\to V(-1)\to 0$ one deduces
    \begin{align}
    \label{eq:truncatedDual}
        \rmH^{i,i}(\Lambda^\bullet(V),L_2(2))\cong \rmH^{i,i+2}(\Lambda^\bullet(V),L_2)\cong \ker (\mathbb{S}^{i+2}(V^*)\otimes V^*\to \mathbb{S}^{i+3}(V^*)).
    \end{align}
    This shows by Proposition~\ref{prop:CohomologyAndDuals}, that $J(2)$ is the dual of $L_2(2)$. Now consider the inclusion map $\iota:I\to L_2$. Then by Lemma~\ref{lem:monomorphismDual} implies that the following sequence is exact
    \begin{align*}
        (\Lambda^2(V)/I_2)^*\otimes \mathbb{S}^\bullet(V)\longrightarrow L_2(2)^!\cong J(2)\overset{\iota(2)^!}\longrightarrow I(2)^!\longrightarrow 0
    \end{align*}
    and thus $I(2)^!\cong J(2)/\langle W^*\rangle$. By Proposition~\ref{prop:SimultaniousKoszul} applied with $a=\infty$ and $b=1$ one sees that $\rmH^{i,i+3}(\Lambda^\bullet(V),I)=0$ for all $i$ if and only if for all $j>4$
    \begin{align*}
        0=\rmH^{2,j}(\mathbb{S}^\bullet(V^*),I(2)^!)\cong \rmH^{2,j}(\mathbb{S}^\bullet(V^*),J/\langle W^*\rangle )
    \end{align*}
    as claimed.
\end{proof}
\begin{rem}
    Proposition~\ref{prop:ThreeCohomologyGroupsAreEnough} yields an algorithmic way to check the conditions of Theorem~\ref{thm:WeakKoszulBPC} in finite time. Several computer algebra systems are capable of computing graded cohomology groups in reasonable time.

    We implemented this method to search for examples of ideals in $\Lambda^\bullet(V)$ satisfying the conditions of Theorem~\ref{thm:WeakKoszulBPC}, but that were not Koszul. We were not able to find one until now.
\end{rem}

\section{The Module Koszulity Conjecture for oriented pro-\texorpdfstring{$p$}{} groups of elementary type}
\label{sec:ModuleKoszulityForElementaryType}

\begin{defn}
    A Demushkin group is a finitely generated pro-$p$ group $G$ of cohomological dimension $2$ (i.e. $\rmH^i(G,\FF_p)=0$ for $i>2$) with $\rmH^2(G,\FF_p)\cong \FF_p$ such that the cup product induces a non-degenerate bilinear pairing
    \begin{align*}
        \rmH^1(G,\FF_p)\otimes \rmH^1(G,\FF_p)\to \rmH^2(G,\FF_p).
    \end{align*}
\end{defn}
It turns out that there is exactly one orientation $\eth$ for a Demushkin group that turns $(G,\eth)$ into a Kummerian oriented pro-$p$ group (see \cite{QuadrelliWeigel2020}*{Proposition 5.2}).

We can now define the class of oriented pro-$p$ groups of elementary type.
\begin{defn}
    \label{def:ElementaryType}
    Let $\mathcal{ET}_p$ be the smallest class of oriented pro-$p$ groups satisfying the following conditions:
    \begin{enumerate}
        \item $\mathcal{ET}_p$ contains $\ZZ_p$ with any orientation $\theta:\ZZ_p\to \ZZ_p^\times$;
        \item $\mathcal{ET}_p$ contains all Demushkin groups $G$ with their canonical orientation $\eth$ making $(G,\eth)$ into a Kummerian oriented pro-$p$ group;
        \item \label{it:freeprod} if $(G_1,\theta_1)$, $(G_2,\theta_2)\in \mathcal{ET}_p$, then $(G_1*_pG_2,\theta_1*\theta_2)$ is also contained in $\mathcal{ET}_p$;
        \item \label{it:semidirectprod} if $(G,\theta)$ is in $\mathcal{ET}_p$ and $A$ is a finitely generated free abelian pro-$p$ group, then $(A\rtimes_\theta G,\theta\circ \pi_2)$ is also contained in $\mathcal{ET}_p$.
    \end{enumerate}
    An oriented pro-$p$ group in $\mathcal{ET}_p$ is said to be of \emph{elementary type}.
\end{defn}
Our goal is to show the following theorem:
\begin{thm}
    \label{thm:ElementaryTypeKoszul}
    Let $(G,\theta)$ be a torsion-free oriented pro-$p$ group of elementary type, then $(\ker \psi_G^\bullet)(2)$ is Koszul.
\end{thm}
The proof is structured in multiple steps. It is clear that $\ZZ_p$ with any torsion-free orientation satisfies the theorem, since $\rmH^\bullet(\ZZ_p,\FF_p)\cong \Lambda^\bullet(\rmH^1(\ZZ_p,\FF_p))$. Next, we prove that the statement is true for Demushkin groups, and we show that the condition is stable under the operations (\ref{it:freeprod}) and (\ref{it:semidirectprod}) of Definition~\ref{def:ElementaryType}. 

The condition that $(G,\theta)$ is torsion-free only poses a restriction in the case where $p=2$. The image of $\theta_1*\theta_2$ is $\langle \im(\theta_1),\im(\theta)_2\rangle$, and thus contained in $1+4\ZZ_2$ if and only if the images of both $\theta_1$ and $\theta_2$ are contained in $1+4\ZZ_2$. Furthermore, (\ref{it:semidirectprod}) preserves the image of $\theta$. Thus it is sufficient to start in any case with torsion-free oriented groups, when proving the property for the \lq\lq building blocks\rq\rq\, of groups of elementary type.

\begin{prop}
\label{prop:DemushkinKoszul}
    Let $(G,\eth)$ be a Demushkin group whose natural orientation is torsion-free; then the module $(\ker \psi^\bullet)(2)$ is Koszul over $\Lambda^\bullet(\rmH^1(G,\FF_p))$.
\end{prop}
\begin{proof}
    Set $V:=\rmH^1(G,\FF_p)$ and define $L_k$ as in the proof of Proposition \ref{prop:ThreeCohomologyGroupsAreEnough}. We get a short exact sequence of $\Lambda^\bullet(V)$-modules.
    \begin{equation*}
        \begin{tikzcd}
            0\arrow[r]&I:=(\ker \psi^\bullet)(2)\arrow[r]&L_2(2)\arrow[r]&\FF_p \arrow[r]&0
        \end{tikzcd}
    \end{equation*}
    Since $L_2(2)$ and $\FF_p$ are both Koszul modules, the long exact sequence for $\rmH^*(\Lambda^\bullet(V),\bl)$ shows that $\rmH^{i}({\Lambda^\bullet(V)},I)$ is concentrated in degrees $i$ and $i+1$. To show that we have vanishing in degree $i+1$, we use the following diagram with exact top row:
    \begin{equation*}
        \begin{tikzcd}[column sep=small]
            0\arrow[r]&\rmH^{i,i+1}(\Lambda^\bullet(V),I)\arrow[r]&\rmH^{i+1,i+1}(\Lambda^\bullet(V), \FF_p)\arrow[r]\arrow[d,equal]&\rmH^{i+1,i+1}(\Lambda^\bullet(V),L_2(2))\arrow[d,hookrightarrow,"(\ref{eq:truncatedDual})"]\\
            &&\mathbb{S}^{i+1}(V^*)\arrow[r,"\alpha_{i+1}"]&\mathbb{S}^{i+2}(V^*)\otimes V^*
        \end{tikzcd}
    \end{equation*}
    Thus we can conclude that $I$ is Koszul if and only if $\alpha_{i+1}$ is injective for all $i$.
    
    Since the map $\alpha$ comes from taking quadratic duals of $L_2(2)\to \FF_p$, it is induced by the dual of the multiplication map $\Lambda^2(V)\to \FF_p$. Using \cite[Proposition 3.9.16]{NSW2008}, we can choose a basis $\chi_1,...,\chi_d$ of $V$, such that
    \begin{align*}
        1=\chi_i\cup \chi_{i+1}=-\chi_{i+1}\cup \chi_i\qquad \text{for all }i=1,...,d-1
    \end{align*}
    and the product $\chi_i\cup \chi_j$ is $0$ in all other cases. If we denote by $x_1,..,x_d$ the dual basis of $V^*$. Then we can write $\alpha_{i+1}$ explicitly as
    \begin{align*}
        \alpha_{i+1}:\mathbb{S}^{i+1}(V^*)\to \mathbb{S}^{i+2}(V^*)\otimes V^*, \quad f\mapsto \sum_{i=1}^{d-1}(x_if)\otimes x_{i+1}-(x_{i+1} f)\otimes x_i.
    \end{align*}
    It is easy to see that $\alpha_{i+1}$ is injective by composing it with $\id \otimes \chi_1$, where we interpret $\chi_1$ as an element of $(V^*)^*$. Thus, $\alpha_{i+1}$ is injective and $I$  Koszul.
\end{proof}

\begin{prop}
\label{prop:FreeProductKoszul}
    Let $(G_1,\theta_1)$ and $(G_2,\theta_2)$ be Kummerian, torsion-free, oriented pro-$p$ groups with each $G_i$ being $\rmH^\bullet$-quadratic. Set $(G,\theta)=(G_1,\theta_1)*_p(G_2,\theta_2)$. Assume that $(\ker \psi^\bullet_{G_k})(2)$ is also a Koszul module over $\Lambda^\bullet(\rmH^1(G_k,\FF_p))$ for $k=1,2$, then $(\ker \psi_G^\bullet)(2)$ is a Koszul module over $\Lambda^\bullet(\rmH^1(G,\FF_p))$.
\end{prop}
\begin{proof}
    For abbreviation, we set $\Lambda:=\Lambda^\bullet(\rmH_1(G,\FF_p))$, $\rmH^\bullet(G_k):=\rmH^\bullet(G_k,\FF_p)$, and $\Lambda_k:=\Lambda^\bullet(\rmH^1(G_k))$ for $k=1,2$. Then we have $\Lambda=\Lambda_1\otimes^{-1}\Lambda_2$ by Construction~\ref{con:TensorProductOfAlgebras}. Using the exact sequence
    \begin{equation*}
        \begin{tikzcd}[row sep=small]
            0\arrow[r]&\ker \psi_{G_i}^\bullet \arrow[r]& \Lambda_i\arrow[r]& \rmH^\bullet(G_i)\arrow[r]&0
        \end{tikzcd}
    \end{equation*}
    we see that the Koszulity of $(\ker \psi_{G_i}^\bullet)(2)$ implies that $\rmH_j(\Lambda_1,\rmH_\bullet(G_i))$ is concentrated in degree $j+1$ for all $j>0$. 

    By \cite[Theorem 4.1.4]{NSW2008}, we get a short exact sequence of $\Lambda$-modules:
    \begin{equation*}
        \begin{tikzcd}
            0\arrow[r]&\rmH^\bullet(G_1*_p G_2)\arrow[r]&\rmH^\bullet(G_1)\oplus \rmH^\bullet(G_2) \arrow[r]&\FF_p \arrow[r]&0
        \end{tikzcd}
    \end{equation*}
    Because $\FF_p$ is a Koszul $\Lambda$-module, the long exact sequence for $\rmH_*(\Lambda,\bl)$ yields isomorphisms for $k>j+1$.
    \begin{align*}
        \rmH_{j,k}(\Lambda,\rmH^\bullet(G_1*_pG_2))\cong \rmH_{j,k}(\Lambda,\rmH^\bullet(G_1))\oplus \rmH_{j,k}(\Lambda,\rmH^\bullet(G_2)).
    \end{align*}
    We show that each of these groups is zero for $j>0$, from which we conclude that $(\ker \psi^\bullet)(2)$ is Koszul.
    
    The $\Lambda$-modules $\rmH^\bullet(G_i)$ are isomorphic to $\rmH^\bullet(G_1)\otimes^{-1}\FF_p$ resp. $\FF_p\otimes^{-1}\rmH^\bullet(G_2)$. We only study the case for $i=1$, the other is analogous. We can apply Proposition~\ref{prop:CohomologyOfTensorProduct} and get
    \begin{align*}
        \rmH_j(\Lambda,\rmH^\bullet(G_1))\cong \bigoplus_{s+t=j}\rmH_s({\Lambda_1},\rmH^\bullet(G_1))\otimes \rmH_t({\Lambda_2},\FF_p)
    \end{align*}
    The vector space $\rmH_s({\Lambda_1},\rmH^\bullet(G_1))$ is concentrated in degree $s+1$ for $s>0$ and $0$ if $s=0$. The vector space $\rmH_t({\Lambda_2},\FF_p)$ is concentrated in degree $t$. Therefore, the graded vector space $\rmH_{j,k}(\Lambda,\rmH^\bullet(G_1))$ is zero for $k>j+1$, implying that $\rmH_{i,j}(\Lambda,\ker\psi_G^\bullet)=0$ for $j>i+2$ and therefore $(\ker \psi_G^\bullet)(2)$ is Koszul.
\end{proof}
\begin{prop}
\label{prop:SemiDirectKoszul}
    Let $(G_0,\theta_0)$ be a Kumerian torsion-free, oriented pro-$p$ group and $A$ be a finitely generated free abelian pro-$p$ group. Assume that $(\ker \psi_{G_0}^\bullet)(2)$ is a Koszul module over $\Lambda^\bullet(\rmH^1(G_0,\FF_p))$. Then the same is true for $(G,\theta):=(A\rtimes_{\theta_0}G,\theta_0\circ \pi)$.
\end{prop}
\begin{proof}
    By \cite[Proposition 5.8]{MinacPasiniQuadrelliTan2021} we have $\rmH^\bullet(G,\FF_p)\cong \rmH^\bullet(G_0,\FF_p)\otimes^{-1}\Lambda^\bullet(V)$ for $V:=A/pA$. We get $\Lambda:=\Lambda^\bullet(\rmH^1(G,\FF_p))\cong \Lambda_0\otimes^{-1} \Lambda^\bullet(V)$. Again, by Proposition~\ref{prop:CohomologyOfTensorProduct} one concludes
    \begin{align*}
        \rmH_k(\Lambda,\rmH^\bullet(G,\FF_p))&\cong \rmH_k({\Lambda_0\otimes^{-1} \Lambda^\bullet(V)},\rmH^\bullet(G_0,\FF_p)\otimes^{-1}\Lambda^\bullet(V))\\
        &\cong \bigoplus_{s+t=k} \rmH_s(\Lambda_0,\rmH^\bullet(G_0,\FF_p))\otimes \rmH_t({\Lambda^\bullet(V)},\Lambda^\bullet(V))\\
        &\cong \rmH_k(\Lambda_0,\rmH^\bullet(G_0,\FF_p)).
    \end{align*}
    Thus also $\rmH_k(\Lambda_0,\ker \psi_{G_0}^\bullet)\cong \rmH_k(\Lambda,\ker\psi_G^\bullet)$, which implies the desired statement.
\end{proof}

Combining the propositions \ref{prop:DemushkinKoszul}, \ref{prop:FreeProductKoszul}, and \ref{prop:SemiDirectKoszul} yields the desired proof of Theorem~\ref{thm:ElementaryTypeKoszul}.

\begin{rem}
    We have even shown the validity of the Module Koszulity Conjecture~1 for more general fields, than the ones, whose maximal pro-$p$ Galois group is of elementary type, by not restricting to the finitely generated case in Proposition~\ref{prop:FreeProductKoszul} and \ref{prop:SemiDirectKoszul}. For example, if $(\KK,v)$ is a complete discretely valued field, such that the residue field satisfies the Module Koszulity Conjecture~1, then the same is true for $\KK$. This applies for example to $\LL(\!(t)\!)$. This yields a new proof of \cite{Positselski2014}*{Theorem 1 (2)}.
\end{rem}


\subsection*{Acknowledgments} I would like to thank my PhD-advisor Thomas Weigel of the Università degli Studi di Milano-Bicocca for posing this problem to me and supporting me during my research. I owe the beautiful counter example in \ref{exmp:IdealNotQuadratic} to Simone Blumer and I am thankful for various helpful discussions with him, my PhD-advisor Christian Maire, and Claudio Quadrelli. 

I would like to thank an anonymous referee of an earlier version of this paper for his constructive suggestions, helping to improve the article.

I am a member of the Gruppo Nazionale per le Strutture Algebriche, Geometriche e le loro Applicazioni (GNSAGA), which is part of the Istituto Nazionale di Alta Matematica (INdAM). 

\bibliography{literature}

\end{document}